\documentclass [10pt] {article}
\title{{\bf An Extension of the Classical Gauss Series-product Identity by Fermionic Construction of $\widehat{\mathfrak{sl}}_n$  }}
\author{Tomislav \v{S}iki\' c\\
        Department of Mathematics\\
        University of Zagreb\\
       Bijeni\v{c}ka 30\\
      10000 Zagreb, Croatia\\
    E-mail: sikic@math.hr
        }
\date{ 2007.}
\usepackage{eufrak}
\DeclareSymbolFont{AMSb}{U}{msb}{m}{n}
\DeclareSymbolFontAlphabet{\Bb}{AMSb}

\newtheorem{theorem}{Theorem}[section]

\newtheorem{remark}[theorem]{Remark}

\newtheorem{proposition}[theorem]{Proposition}

\def\proofname{Proof}
\newenvironment{proof}[1][\proofname]{
\par \normalfont \trivlist \itemindent\parindent
  \item[\hskip\labelsep\scshape #1{.}]
\ignorespaces}{  \qed\endtrivlist}

\let\qed\Box

\begin{document}
\maketitle
\begin{abstract}
The main result of this paper is two infinity classes  of
series-product identities which is based on classical Gauss identity
and two different interpretations of character formula for
irreducible highest weight modules of affine Lie algebras.
\end{abstract}

\section{INTRODUCTION}
\def\theequation{\thesection.\arabic{equation}}
\setcounter{equation}{0}

 It is well known that  celebrated Macdonald
identities (and especially the Jacobi triple product identity) are
nothing else than the denominator identity for affine Kac-Moody Lie
algebras. Moreover, some specializations of the denominator identity
give interesting series-product identities. For instance, the
following classical identities
\begin{eqnarray}
\varphi (q) & = & \sum_{n\in \Bb Z} (-1)^{n}q^{(3n^{2}+n)/2}\
(Euler)\ ,\nonumber\\
\varphi (q)^{3} & = & \sum_{n\in \Bb Z} (4n+1)q^{2n^{2}+n} \ (Jacobi)\ ,\nonumber\\
\frac{\varphi (q)^{2}}{\varphi (q^{2})} & = & \sum_{n\in \Bb Z}
(-1)^{n}q^{n^{2}}\ (Gauss)\ ,\nonumber
\end{eqnarray}
and in particular
\begin{equation}\label{1}
\frac{\varphi {(q^2)}^{2}}{\varphi (q)}  =  \sum_{n\in \Bb Z}
q^{2n^{2}+n }\ (Gauss)\ ,
\end{equation}
where $\varphi(q)=\prod_{j\geq 1}(1-q^{j})$ is Euler's product
function, can be also expressed following the same approach (see
\cite{Kac}, Exercise 12.4, pp. 241). It is also quite interesting
that the classical Gauss identity (\ref{1}) arises  from the first
concrete computations of characters of nontrivial modules for affine
Lie algebras (see \cite{FL}).
\\
As Victor Kac writes in his book \cite{Kac} (pp. 216) the basic idea
of this approach is very simple:
 "one gets an interesting identity by computing the
 character of integrable representation in two different ways and
 equating the results. In particular, Macdonald identities are
 deduced via trivial representation."
\\
Following  Kac's inventive consideration the central object of the
observation in this article will be  a character formula of an
irreducible highest weight module $L(\Lambda)$ of the affine Lie
algebra $\widehat{\mathfrak{sl}}_n$, where $\Lambda$ can be any
fundamental weight, not just a trivial representation. The above
observation results with two infinity classes  of series-product
identities. If we denote  by $m$ the arbitrary positive integer  and
by $\kappa$ the following polynomial of the several variables
\[
\kappa (k_1,...,k_{4m-1})= k_1^2+k_2^2+\cdots +k_{4m-1}^2
-k_1k_2-k_2k_3-\cdots-k_{4m-2}k_{4m-1}\
\]
than the first class of series-product identities looks like
\begin{equation}\label{2}
\frac{\varphi(q^{4m-1})^{4m-1}\varphi
(q^{2m})^2}{\varphi(q)\varphi(q^m)}  = \sum_{k_1,...,k_{4m-1}\in\Bb
Z} q^{(4m-1)\kappa(k_1,...,k_{4m-1})+lin (k_1,...,k_{4m-1})} \ ,
\end{equation}
where
\[
lin (k_1,...,k_{4m-1})=(2m-1)k_1 -k_2-\cdots -
k_{3m-1}+(4m-2)k_{3m}-k_{3m+1}-\cdots -k_{4m-1}\ .
\]
The another class
\begin{equation}\label{3}
\frac{\varphi(q^{3m})^{4m}\varphi
(q^{2})^2}{\varphi(q)^2\varphi(q^3)} = \sum_{k_1,...,k_{4m-1}\in\Bb
Z} q^{3m\kappa(k_1,...,k_{4m-1})+\widetilde{lin} (k_1,...,k_{4m-1})}
\end{equation}
holds for
\[
\widetilde{lin} (k_1,...,k_{4m-1})=-3k_1-\cdots -3k_{m-1}
+(3m-2)k_{m}-k_{m+1}-\cdots -k_{4m-2}+(3m-1)k_{4m-1}\ .
\]
Since the identity (\ref{1}) will be essentially involved in the
proof of  mentioned classes (\ref{2}) and (\ref{3}) we can interpret
it as an extension of this classical Gauss identity.

As we say above, we looked at the mentioned object from two
different points of view. One point of view is  based on the
\emph{character formula}
\begin{equation}\label{4}
ch\ {L(\Lambda)}=e^{\frac{1}{2}{|\Lambda|}^2\delta}\frac
{\sum_{\gamma\in\overline{Q}+\overline{\Lambda}}e^{\Lambda_0+\gamma-\frac{1}{2}|\gamma|^{2}\delta}}{\prod_{j\geq
1}(1-e^{-j\delta})^{mult\ j\delta}}\ .
\end{equation}
for any dominant integral weight $\Lambda$ in the special case of
the affine Lie algebras of type $A_l^{(1)}$, $D_l^{(1)}$,
$E_l^{(1)}$ (see \cite{Kac} or \cite{K1} and \cite{KP1}).
\\
Another point of view is based on a \emph{boson-fermionic
realization of $L(\Lambda)$ for affine Lie algebra}
$\widehat{\mathfrak{gl}}_n$ (see \cite{KL}), parameterized by a
partition of a number $\underline{n} = \{n_1 , n_2 , \cdots ,
n_r\}$. The corresponding character formula for the affine Lie
algebra $\widehat{\mathfrak{sl}}_n$ in the original notation
\cite{KL} is
\begin{equation}\label{5}
Trace_{L(\Lambda_k)}(q) = q^{const}\frac{\varphi
(q)}{\prod_{i=1}^{r} \varphi(q^{1/n_i})} \sum_{k_1+k_2+\cdots
 +k_r=k}q^{\frac{1}{2}(\frac{k_1^2}{n_1}+\frac{k_2^2}{n_2}+ \cdots +
 \frac{k_r^2}{n_r})}\ .
\end{equation}

The connection between two different points of view is made by a
particular specialization ${\cal F}_{\bf{s}} : {\Bb
C}[[e^{\alpha_0},e^{\alpha_1},e^{\alpha_2},..., e^{\alpha_l}]]
\rightarrow {\Bb C}[[q]]$ of type ${\bf{s}}=(s_0,s_1,s_2,...,s_l)$
which satisfies the following formula
\begin{equation}\label{6}
{\cal F}_{\bf{s}} (ch\ L(\Lambda_k))=
q^{const}Trace_{L(\Lambda_k)}(q^N)\ .
\end{equation}
The $n$-tuple  ${\bf{s}}$ and positive integer $N$ are generated (we
shall provide details later) from the same partition
$\underline{n}$.

In this article we would like to emphasize the importance of
(\ref{6}). Observe that both sides of (\ref{6}) contain infinite
sums (denoted by $\sum_{left}$ and $\sum_{right}$) dependent on
significantly different sizes of the set of indices. On the left
hand side we have $n-1$ indices, while on the right hand side we
have $r-1$. Very often we have the case (which is of particular
interest to us) when
\[n\gg r\ .\]
In this particular case we obtain a significant reduction of the sum
$\sum_{left}$
 by the sum $\sum_{right}$ modulo some fraction of   Euler's product functions $\varphi(q)$.
Hypothetically, if we recognize, in particulary  cases,  that  the
sum $\sum_{right}$ is  a one side of the some well known
series-product identity  we can express the sum $\sum_{left}$, which
depends on arbitrary indices using the fraction of  Euler's product
functions $\varphi(q)$.
\\
Moreover, we would like to argue that our approach provides an
algebraic method to reveal numerous important series-product
identities; studied mainly by number-theorists. This is one of the
main points of our article. We illustrate the power of  our method
by constructing two infinite classes of series-product identities;
both are based on the classical Gauss series-product identity
(\ref{1}).

 Finally, let us also mention that these two classes have only
one common element which was the starting point of our research. We
shall conclude this introduction by providing details on this
particular example. This example provides a  review of above
observations and in the same time is a motivation for the
construction of the mentioned two series-product identity classes.
\\
We also believe that this example will guide our reader through the
main body of the article.

Let $\mathfrak{g}$ be a simple Lie algebra of the type $A_3$, i.e.,
$\mathfrak{g} ={\mathfrak{sl}}_4$. For the corresponding affine Lie
algebra $\hat\mathfrak{g} =\widehat{\mathfrak{sl}}_4$, the partition
$\underline{4}= \{1,3\}$ and the fundamental weight $\Lambda_3$ the
character formula (\ref{4}) looks like
\[
ch_{L(\Lambda_3)}=e^{\frac{1}{2}{|\Lambda_3|}^2\delta}\frac
{\sum_{k_1,k_2,k_3\in \Bb Z}e^{\Lambda_0+(k_1+1/4)\alpha_1 +(k_2
+2/4)\alpha_2+(k_3+3/4)\alpha_3-\frac{1}{2}|\gamma|^{2}\delta}}{\prod_{n\geq
1}(1-e^{-j\delta})^3}
\]
where
\[
\frac{1}{2}|\gamma|^{2}=(k_1+\frac{1}{4})^{2}+(k_2+\frac{2}{4})^{2}+(k_3+\frac{3}{4})^{2}-
(k_1+\frac{1}{4})(k_2+\frac{2}{4})-(k_2+\frac{2}{4})(k_3+\frac{3}{4})
\]
and
\[
\delta=\alpha_0+\alpha_1+\alpha_2+\alpha_3
\]
is  the corresponding  imaginary root. \\
For above settings $(\widehat{\mathfrak{sl}}_4,\ \underline{4}=
\{1,3\},\ \Lambda_3)$  the trace formula (\ref{5})  has the
following form
\[
Trace_{L(\Lambda_3)}(q) = q^{const}\frac{\varphi (q)}{
\varphi(q^{1/1})\varphi(q^{1/3})}
\sum_{k_1+k_2=3}q^{\frac{1}{2}(\frac{k_1^2}{1}+\frac{k_2^2}{3})}\ .
\]
Now using the substitution of variables $k_2= 3-k_1$ and classical
Gauss identity (\ref{1}) in the sum
\[
\sum_{k_1+k_2=3}q^{\frac{1}{2}(\frac{k_1^2}{1}+\frac{k_2^2}{3})}
\]
we have
\[
Trace_{L(\Lambda_3)}(q) =
q^{const}\frac{{\varphi(q^{2/3})}^2}{{\varphi(q^{1/3})}^2} \ .
\]
\\
As outlined above, we now apply a particular specialization, in this
case the specialization ${\cal F}_{\bf{s}} :{\Bb
C}[[e^{\alpha_0},e^{\alpha_1},e^{\alpha_2},
e^{\alpha_3}]]\rightarrow {\Bb C}[[q]] $  defined by parameters
\[{\bf{s}}=(s_0,s_1,s_2,s_3)= (2,-1,1,1)\]
 which are also generated by the
partition $\underline{4}=\{1,3\}$. Now, using (\ref{6}) for $N=3$ we
obtain
\[
\frac{{\varphi(q^{2})}^2
{\varphi(q^{3})}^3}{{\varphi(q)}^2}=\sum_{k_1,k_2,k_3\in \Bb Z}
q^{3(k_1^{2}+k_2^{2}+k_3^{2}-k_1 k_2-k_2 k_3)+k_1-k_2+2k_3} \ ,
\]
which seems to be a new series-product identity extended from the
classical Gauss identity (\ref{1}).

\section{THE NOTATION AND BASIC SETTINGS }
\def\theequation{\thesection.\arabic{equation}}
\setcounter{equation}{0}
Let $\mathfrak{g}={\mathfrak{sl}}_n$ be a
simple Lie algebra defined for the  Dynkin diagram $A_l$, where
$l=n-1$.  Denote by $\mathfrak{h}$ the corresponding Cartan
subalgebra and by
\[\Delta =(\alpha_1,...,\alpha_l)\]
the basis of the root system ${\cal R}$($\subset\mathfrak{h}^{*} $).
Besides, by $\theta$ we denote the highest root of the root system
${\cal R}$ . It is well known that
\begin{eqnarray}\label{2.1}
&{\cal R}  =  \{\pm(\varepsilon_i -\varepsilon_j) | 1\leq i<j\leq l+1\}&\nonumber\\
&\alpha_1 =  \varepsilon_1 -\varepsilon_2 ,\alpha_2 = \varepsilon_2
-\varepsilon_3 ,...,\alpha_{l-1}= \varepsilon_{l-1} -\varepsilon_l,
\alpha_l =  \varepsilon_l-\varepsilon_{l+1}&\
\\
&\theta  =  \alpha_1 +\alpha_2 + \alpha_3 +\cdots +\alpha_l =
\varepsilon_1 -\varepsilon_{l+1}&\nonumber\ .
\end{eqnarray}
Let
\[
\mathfrak{g}=\mathfrak{h}\oplus\bigoplus_{\alpha\in {\cal
R}}\mathfrak{g_{\alpha}}
\]
be a root space decomposition of the  simple Lie algebra. Denote by
$x_{\alpha}\in \mathfrak{g}_{\alpha}$ the root vector which
satisfies
\[[x_{\alpha},x_{-\alpha}]=\alpha^{\vee}\]
for the coroot $\alpha^{\vee}$.\\
Let
\[\hat\mathfrak{g}=\mathfrak{g}\otimes\Bb{C}[t,t^{-1}]\oplus(\Bb{C}c+\Bb{C}d)
\]
and write $x(i)=x\otimes t^i$ for $x\in \mathfrak{g}$ and $i\in
\Bb Z$. Then $\hat\mathfrak{g}=\widehat\mathfrak{sl}_n$ is an
affine Lie algebra with
\[
[x(i),y(i)]= [x,y](i+j)+i\delta_{i+j,0}( x\mid y)\ ,
\]
where $( x\mid y)$ is the Killing form for the simple Lie algebra
$\mathfrak{g}$, $c$ being a central element
\[
c=\sum_{i=0}^l \alpha_i^{\vee}
\]
 and $d$ a scaling
element
\[[d,x(i)]=ix(i)\ .\]
The Cartan subalgebra of
$\hat\mathfrak{g}$ is given by
\[\hat\mathfrak{h} = \mathfrak{h}\oplus (\Bb{C}c+\Bb{C}d)\ .\]
The corresponding Dynkin diagram is of type $A_l^{(1)}$ and related
numerical labels are
\begin{equation}\label{2.2}
a_{A^{(1)}_l}  =  (1,1,1,...,1,1,1) \ .
\end{equation}
We denote by $\delta$ the linear functional on CSA
$\hat\mathfrak{h}$ defined by
\[
\delta\mid_{\mathfrak{h}\oplus \Bb{C}c} = 0 \hskip1cm \langle \delta
,d\rangle =( \delta\mid d)= 1 .
\]
The affine Lie algebra root system $\hat{\cal
R}$($\subset\hat\mathfrak{h}^{*}$) is composed of the real and
imaginary roots
\[
\hat{\cal R} = \hat{\cal R}^{Re}\cup \hat{\cal
R}^{Im}=\{\alpha+n\delta|\alpha \in {\cal R}, n\in\Bb Z
\}\cup\{n\delta|n\in \Bb Z \setminus \{0\}\}\ .
\]
If we denote by $\alpha_0$ the following root
\begin{equation}\label{2.3}
\alpha_0 = -\theta +\delta
\end{equation}
then
\[
\hat\Delta =(\alpha_0,\alpha_1,...,\alpha_l)
\]
form the basis of the root system $\hat{\cal R}$ and $Q=\sum_{i=0}^l
\Bb Z\alpha_i$ is the corresponding root lattice.  \\
The imaginary root $\delta$, spanned in above basis, due to
(\ref{2.2}) and (\ref{2.3}), looks like
\begin{equation}\label{2.4}
\delta=\sum_{i=0}^l a_i\alpha_i=\sum_{i=0}^l \alpha_i\ .
\end{equation}
It is well known that all affine Lie algebras are a Kac-Moody
algebras $\mathfrak{g}(A)$ for a generalized Cartan matrix $A$ of
the corank one.  Since  the affine Lie algebra
$\widehat{\mathfrak{sl}}_n$ is also the Kac-Moody algebra then for
every $\Lambda\in \mathfrak{h}^*$ the irreducible highest-weight
module $L(\Lambda)$ is uniquely defined . The number
\begin{equation}\label{2.5}
\Lambda (c)=\langle\Lambda, c\rangle
\end{equation}
is called the level of the weight $\Lambda$ or of the module
$L(\Lambda)$.\\
Denote by $P(\Lambda)$ the set of all weights of the module
$L(\Lambda)$ and by $mult\ \lambda$ the multiplicity of $\lambda\in
P(\Lambda)$. The  set
\[
P = \{\lambda\in\hat\mathfrak{h}^{*}|\lambda(\alpha_i^{\vee})\in\Bb
Z,\ i=0,1,...,n-1\}
\]
 is called the weight lattice and  weights from $P$ are
called integral weights. Integral weights from
\[
P_+ =  \{\lambda\in P|\lambda(\alpha_i^{\vee})\geq 0,\
i=0,1,...,n-1\}
\]
are called
dominant. The weight lattice contains the root lattice $Q$ and it is
clear that
\[P(\Lambda)\subset P\]
if $\Lambda \in P$. Besides, the irreducible highest-weight
$\mathfrak{g}(A)$-module  is integrable if and only if $\Lambda \in
P_+$. The fundamental weights $\Lambda_i$ for $i=0,1,...,n-1$ are
defined by
\begin{equation}\label{2.6}
\Lambda_i(\alpha_j^{\vee})=\delta_{ij}\ , \ j=0,1,...,n-1\ \ and \ \
\Lambda_i(d)=0\ .
\end{equation}
It is obvious that fundamental weights are always dominant.

\section{THE FIRST POINT OF VIEW}
\def\theequation{\thesection.\arabic{equation}}
\setcounter{equation}{0}

For a subset  $S$ of $\hat\mathfrak{h}^{*}$ denote by $\overline{S}$
the orthogonal projection of $S$ on $\mathfrak{h}^{*}$ by the
extension of  the Killing form from the simple $\mathfrak{g}$ to the
affine Lie algebra $\hat\mathfrak{g}$. Then we have the following
useful formula for $\lambda\in\hat\mathfrak{h}^{*} $
\[
\lambda=\overline{\lambda}+
\lambda(c)\Lambda_0+(\lambda\mid\Lambda_0)\delta
\]
where $\Lambda_0$ is the fundamental weight.\\
Especially for the $\widehat{\mathfrak{sl}}_n$ we have that
$\overline{Q}=\sum_{i=1}^l \Bb Z\alpha_i$ and
\[
\Lambda_i=\Lambda_0+\overline{\Lambda}_i
\]
where $\overline{\Lambda}_0=0$ and
$\overline{\Lambda}_1,...,\overline{\Lambda}_{n-1}$ are the
fundamental weights of simple Lie algebra $\mathfrak{sl_n}$.
\newline
Since ${\Lambda}_i\in P_+$ then for all fundamental weights the
irreducible highest-weight\\ $\widehat{\mathfrak{sl}}_n$-module
$L({\Lambda}_i)$ is integrable.\\
Besides, from (\ref{2.5}) and (\ref{2.6}) follow that all
$\widehat{\mathfrak{sl}}_n$-module $L({\Lambda}_i)$ are level one
modules. Using the notation $P_+^1$ for the set of all level one
dominant integral weights we can write
\begin{equation}\label{3.1}
\Lambda_i \in P_+^1\ \ \forall i=0,1,...,n-1\ .
\end{equation}
Moreover, when $\Lambda\in P_+^1$ and the type of  Dynkin diagram is
equal to $A_l^{(1)}$, $D_l^{(1)}$ or $E_l^{(1)}$  the following
formula
\begin{equation}\label{3.2}
P(\Lambda) = \{\Lambda_0+\frac{1}{2}|\Lambda|^{2}\delta +\alpha
-(\frac{1}{2}|\alpha|^{2}+s)\delta\mid\alpha\in\overline{\Lambda}+\overline{Q},
s\in {\Bb Z}_+\}
\end{equation}
explicitly describes the weights system $P(\Lambda)$. This result is
proved in \cite{K1} or \cite{FK}. \\
A weight $\lambda\in P(\Lambda)$ is called maximal if
$\lambda+\delta\notin P(\Lambda)$. Denote by $max(\Lambda)$ the set
of all maximal weights of $L(\Lambda)$. For $\lambda\in
max(\Lambda)$ the series
\[
a_{\lambda}^{\Lambda}=\sum_{n=0}^{+\infty} mult_{L(\Lambda)}(\lambda
- n\delta)e^{-n\delta}\
\]
is well defined. Using the result (\ref{3.2}), the theory of the
series $a_{\lambda}^{\Lambda}$ (which has been started by \cite{FL}
and \cite{K1}) and  the work \cite{KP1} (or \cite{Kac}, Ch.12) the
character formula of $L(\Lambda)$ can be written as
\[
ch_{L(\Lambda)}=e^{\frac{1}{2}{|\Lambda|}^2\delta}
a_{\Lambda_0}^{\Lambda_0}\sum_{\gamma\in\overline{Q}+\overline{\Lambda}}e^{\Lambda_0+\gamma-\frac{1}{2}|\gamma|^{2}\delta}
=e^{\frac{1}{2}{|\Lambda|}^2\delta}\frac
{\sum_{\gamma\in\overline{Q}+\overline{\Lambda}}e^{\Lambda_0+\gamma-\frac{1}{2}|\gamma|^{2}\delta}}{\prod_{n\geq
1}(1-e^{-n\delta})^{mult\ n\delta}}
\]
Since the Dynkin diagram of affine Lie algebra
$\widehat\mathfrak{sl}_n$ is equal to ${{A}}_l^{(1)}$ and all
fundamental weights $\Lambda_k$ are level one  integral dominant
weights (see \ref{3.1}) then it is obvious that the above character
formula  is appropriate for $\widehat\mathfrak{sl}_n$. Finally, we
finish this exposition  by the character formula
\begin{equation}\label{3.13}
ch_{L(\Lambda_k)}=e^{\frac{1}{2}{|\Lambda_k|}^2\delta}\frac
{\sum_{\gamma\in\overline{Q}+\overline{\Lambda_k}}e^{\Lambda_0+\gamma-\frac{1}{2}|\gamma|^{2}\delta}}{\prod_{n\geq
1}(1-e^{-n\delta})^{mult\ n\delta}}
\end{equation}
for the irreducible integrable  highest weight
$\widehat{\mathfrak{sl}}_n$-module $L({\Lambda}_k)$ where
$k=0,1,...,l$.

\section{THE SECOND POINT OF VIEW}
\def\theequation{\thesection.\arabic{equation}}
\setcounter{equation}{0}

Many of the vertex operator constructions of integrable highest
weight representations are based on an inequivalent Heisenberg
subalgebras. The inequivalent Heisenberg subalgebras, as conjugacy
classes in $S_n$, are parametrized by partitions of $n$. Denote by
$\underline{n}=\{n_1,n_2,...,n_r\}$ a partition of $n$ where
\[
n_1\leq n_2\leq ... \leq n_r\ .
\]
Moreover, the standard canonical base $\{E_{ij}\mid i,j=1,2,...,n\}$
for $\mathfrak{gl}_n$ is also parametrized  by the partition
$\underline{n}$ (see \cite{KL}). The associated partition of
$n\times n$ matices is then given schematically by:
\[
\left[\begin{array}{c|c|c|c}
   {B_{11}}_{\ \ n_1\times n_1} & {B_{12}}_{\ \ n_1\times n_2} & \cdots  & {B_{1s}}_{\ \ n_1\times n_r} \\ \hline
   {B_{21}}_{\ \ n_2\times n_1} & {B_{22}}_{\ \ n_2\times n_2} & \cdots  & {B_{1s}}_{\ \ n_2\times n_r}  \\ \hline
    \vdots & \vdots& \ddots   & \vdots \\ \hline
    {B_{s1}}_{\ \ n_r\times n_1} & {B_{s2}}_{\ \ n_r\times n_2} & \cdots  & {B_{ss}}_{\ \ n_r\times n_r}
\end{array}\right]
\]
where $B_{ij}$ is a block of size $n_i\times n_j$. With this
blockform in mind  the standard canonical base is remodeled with the
set of matrices
\[
\{E_{pq}^{ij}\mid p=1,...,n_i,\ q=1,...,n_j,\ i,j=1,2,...,s\}
\]
where
\begin{equation}\label{4.4}
E_{pq}^{ij}=E_{n_1+\cdots +n_{i-1}+p,n_1+\cdots +n_{j-1}+q}\ .
\end{equation}
Using just mentioned notation, in \cite{KL} authors give an explicit
vertex operator constructions of level one irreducible integrable
highest weight representation of $\widehat{\mathfrak{gl}}_n$ for all
inequivalent Heisenberg subalgebras (i.e. for all partitions).  The
construction uses multicomponent fermionic fields and yields a
correspondence between bosons (elements of Heisenberg subalgebra)
and fermions. The mentioned construction in addition to \cite{KP2}
results with explicit "q-dimension" trace formula for
$\widehat{\mathfrak{gl}}_n$-module
\[
Trace_{\Lambda_k^{\infty}{\Bb C}^{\infty}}q^{D_0} =
q^{\frac{1}{2}{\mid
H_{\underline{n}}\mid}^{2}}\frac{\sum_{k_1+k_2+\cdots
 +k_r=k}q^{\frac{1}{2}(\frac{k_1^2}{n_1}+\frac{k_2^2}{n_2}+ \cdots +
 \frac{k_r^2}{n_r})}}{{\prod_{i=1}^{r}
\prod_{j\geq 1}(1-q^{\frac{j}{n_i}})}}\ ,
\]
 where  $H_{\underline{n}}$ is
element of standard Cartan subalgebra $\mathfrak{h}$ which satisfies
the following commutation relations
\begin{equation}\label{5.27}
ad\
H_{\underline{n}}(E_{kl}^{ij})=[H_{\underline{n}},E_{kl}^{ij}]=(\frac{l}{n_j}-\frac{k}{n_i}+\frac{1}{2n_i}-\frac{1}{2n_j})E_{kl}^{ij}\
.
\end{equation}
Finally, after restriction to $\widehat\mathfrak{sl}_n$ case (also
\cite{KL}), the corresponding irreducible integrable highest weight
module $L(\Lambda_k)$ have the following  trace formula
\begin{equation}\label{4.20}
Trace_{L(\Lambda_k)}(q) = q^{\frac{1}{2}{\mid
H_{\underline{n}}\mid}^{2}}\prod_{j\geq
1}(1-q^{j})\frac{\sum_{k_1+k_2+\cdots
 +k_r=k}q^{\frac{1}{2}(\frac{k_1^2}{n_1}+\frac{k_2^2}{n_2}+ \cdots +
 \frac{k_r^2}{n_r})}}{{\prod_{i=1}^{r}
\prod_{j\geq 1}(1-q^{\frac{j}{n_i}})}}\ .
\end{equation}

\section{THE CONNECTION BETWEEN STANDPOINTS}
\def\theequation{\thesection.\arabic{equation}}
\setcounter{equation}{0}
 Let ${\bf s}=(s_0,s_1,...,s_n)$ be a
sequence of integers. Then the sequence ${\bf s}$ (under some
assumptions) defines a homomorphism
\[
{\cal F}_{\bf s} : {\Bb
C}[[e^{-\alpha_0},e^{-\alpha_1},e^{-\alpha_2},..., e^{-\alpha_n}]]
\rightarrow {\Bb C}[[q]]
\]
by \[
 {\cal F}_{\bf s}(e^{-\alpha_i})= q^{s_i}\ \ (i=0,1,...,n)\ .
\]
This homomorphism is called the specialization of type ${\bf s}$.
\\
Let $N'$ be the least common multiple of $n_1$,$n_2$,...,$n_r$, then
the integer $N$ is defined by:
\begin{equation}\label{5.20}
N=\left\{\begin{array}{ll}
           N' & if\ N'(\frac{1}{n_i}+\frac{1}{n_j})\in 2{\Bb Z}\ \ \ \forall i,j\\
           2N' & if\ N'(\frac{1}{n_i}+\frac{1}{n_j})\notin 2{\Bb Z}\ \ for\
           a\
           pair\ (i,j)\ .
            \end{array}
        \right.
\end{equation}
Following the boson-fermionic construction for
$\widehat\mathfrak{gl}_n $ from the paper \cite{KL} we can conclude
that the connection between the mentioned two different points of
view is made by particular specialization of type
${\bf{s}}=(s_0,s_1,s_2,...,s_n)$ where ${\bf{s}}$ was parametrized
with the  partition $\underline{n}=\{n_1,n_2,...,n_r\}$ and integer
$N$. In fact we have the following proposition.
\begin{proposition}
Let $\underline{n}=\{n_1,...,n_r\}$ be a partition of $n$. Let
 $N$ be the corresponding integer defined by (\ref{5.20}). For affine Lie
 algebra $\widehat\mathfrak{sl}_n$ and for all fundamental weights
 $\Lambda_k$ $k=0,1,...n-1$ the next equation
\begin{eqnarray}\label{5.24}
&&{\cal F}_{\bf s} (\frac
{\sum_{\gamma\in\overline{Q}+\overline{\Lambda}}e^{\Lambda_0+\gamma-\frac{1}{2}|\gamma|^{2}\delta}}{\prod_{j\geq
1}(1-e^{-j\delta})^{mult\ j\delta}})  \nonumber\\
&=& q^{const}\prod_{j\geq 1}(1-q^{jN})\frac{\sum_{k_1+k_2+\cdots
 +k_r=k}q^{\frac{N}{2}(\frac{k_1^2}{n_1}+\frac{k_2^2}{n_2}+ \cdots +
 \frac{k_r^2}{n_r})}}{{\prod_{i=1}^{r}
\prod_{j\geq 1}(1-q^{\frac{jN}{n_i}})}}
 \end{eqnarray}
holds for
\begin{eqnarray}\label{5.24a}
{\bf{s}}&=&
N(\frac{n_1+n_r}{2n_1n_r},\frac{1}{n_1},...,\frac{1}{n_1},\frac{n_1+n_2}{2n_1n_2}-1,
\frac{1}{n_2},...,\frac{1}{n_2},\nonumber\\
& & \\
& &
\frac{n_2+n_3}{2n_2n_3}-1,...,\frac{1}{n_{r-1}},...,\frac{1}{n_{r-1}},\frac{n_{r-1}+n_r}{2n_{r-1}n_r}-1,\frac{1}{n_r},...,\frac{1}{n_r})\
.\nonumber
\end{eqnarray}
\end{proposition}
\begin{proof}
Since the root subspaces $\hat{\mathfrak g}_\alpha$, for
$\alpha\in{\hat{\cal R}}^{Re}$, are one-dimensional we have unique
$1-1$ correspondence
\begin{eqnarray}
\alpha_0 &\longleftrightarrow &\hat{x}_{-\theta}\otimes t
\nonumber\\
\alpha_i &\longleftrightarrow  &\hat{x}_{\alpha_i}\otimes 1
\hskip1cm i=1,...,l\nonumber\ .
\end{eqnarray}
Hence, we have  correspondence, based on (\ref{2.1}), between the
base $\Delta$ and the element of standard canonical base of
$\mathfrak{gl}_n$ $\{E_{ij}\mid i,j=1,...,n\}$
\begin{equation}\label{5.25}
\alpha_i \longleftrightarrow  E_{i,i+1}\ \ i=1,...,l\ .
\end{equation}
From (\ref{2.3}) and (\ref{2.4}) it is evident
\begin{equation}\label{5.26}
\alpha_0 \longleftrightarrow  E_{n,1}\otimes t\ .
\end{equation}
The  commutation relations $(\ref{5.27})$ for $adH_{\underline{n}}$
 and $E_{kl}^{ij}$ as indexed in (\ref{4.4}) express
degrees of eigenvectors $E_{kl}^{ij}$  by
\begin{equation}\label{5.28}
deg E_{kl}^{ij}=
N(\frac{l}{n_j}-\frac{k}{n_i}+\frac{1}{2n_i}-\frac{1}{2n_j})\ mod N\
.
\end{equation}
From  (\ref{5.28}) and from the exposition of \cite{KL} which lead
to trace formula (\ref{4.20}) follows that eigenvalues for the
adjoint action of $adH_{\underline{n}}$ are pointers for the right
specialization $\bf s$. More precisely, using $1-1$ correspondence
(\ref{5.25}), (\ref{5.26}) the sequence $\bf s$  consists of the
eigenvalues for eigenvectors
\[\{\alpha_0^{\vee},\alpha_1^{\vee},...,\alpha_{n-1}^{\vee}\}=
\{E_{n,1}\otimes t,E_{1,2}\otimes 1,E_{2,3}\otimes
1,...,E_{n-1,n}\otimes 1\}\ .\]
as shown in
\begin{equation}\label{5.31}
{\bf s}=(degE_{n,1}+N,degE_{1,2},degE_{2,3},...,degE_{n-1,n})\ .
\end{equation}
Due to remodeling ($\ref{4.4}$) of the  standard canonical base by
$\{E_{ij}^{pq}\}$ we conclude that sequence ${\bf s}$ (\ref{5.31})
is equal to (\ref{5.24a}), i.e. the equation (\ref{5.24}) holds.
\end{proof}
\begin{remark}
In the paper \cite{K3}  V.G.Kac showed that the automorphisms
\[
\sigma_{\bf\tilde{s}}(e_i )=e^{\frac{2\pi i}{N}\cdot s_i} e_i
\hskip1cm i=0,1,...,l
\]
exhaust all $N$-th order automorphisms of $\mathfrak{g}$. By $\{e_i
\mid i=0,1,...,l\}$ are marked generators of $\mathfrak g$ and
\[{\bf{s}} =(s_0,s_1,...,s_l)\] is a sequence of nonnegative
relatively prime integers. The parameters $s_i$ are  called the {\bf
Kac parameters}. Many of the vertex operator constructions of
integrable highest weight representations and the corresponding
gradations and specializations do not provide Kac parameters (see,
in particular \cite{KL} and \cite{KP2}). Particulary, the sequence
$\bf s$ (\ref{5.24a}) and the associated specialization are
determined by relatively prime integers, but all integers
$N(\frac{n_i+n_{i+1}}{2n_i n_{i+1}}-1)$ are negative. So, the
specialization (\ref{5.24})
is not parametrized by Kac parameters.\\
In the paper \cite{TS2} it is given the exact algorithm for finding
the Kac parameters ${\bf s}^{Kac}$ of the sequence $\bf s$
(\ref{5.24a}) and equation (\ref{5.24}) holds for specialization by
Kac parameters, too.
\end{remark}

\section{THE CLASS $GAUSS(\underline{n}=1+(4m-1),\Lambda_{3m})$}
\def\theequation{\thesection.\arabic{equation}}
\setcounter{equation}{0}
 Introduce the Euler product
\[
\varphi (q) = \prod_{n=1}^{\infty} (1-q^n)\ .
\]
Denote by $\kappa$ the following polynomial of the several variables
\[
\kappa (k_1,...,k_{l})= k_1^2+k_2^2+\cdots +k_{l}^2
-k_1k_2-k_2k_3-\cdots-k_{l-1}k_{l}
\]
where $l=n-1$. It is interesting to notice that above polynomial can
be interpreted by the Killing form in  such a way that
\[
\kappa (k_1,...,k_{l})=\frac{1}{2}(k_1\alpha_1^{\vee}+\cdots
k_{l}\alpha_{l}^{\vee}\mid k_1\alpha_1^{\vee}+\cdots
k_{l}\alpha_{l}^{\vee})\ .
\]
\begin{theorem}
Let $n=4m$ for an arbitrary positive integer $m$.  Then the
following series-product identity
\begin{equation}\label{6.1}
\sum_{k_1,...,k_{4m-1}\in\Bb Z}
q^{(4m-1)\kappa(k_1,...,k_{4m-1})+lin (k_1,...,k_{4m-1})} =
\frac{\varphi(q^{4m-1})^{4m-1}\varphi
(q^{2m})^2}{\varphi(q)\varphi(q^m)}
\end{equation}
holds for
\begin{eqnarray}\label{6.2}
lin (k_1,...,k_{4m-1})&=&(2m-1)k_1 -k_2-\cdots -
k_{3m-1}+(4m-2)k_{3m}\nonumber\\
&&-k_{3m+1}-\cdots -k_{4m-1}\ .
\end{eqnarray}
\end{theorem}
\begin{proof}
First of all, the proof is based on the following setting:
\begin{eqnarray}\label{6.3}
\hat\mathfrak{g}& = & \widehat\mathfrak{sl}_{4m}\nonumber\\
\underline{n} & = & \underline{4m}=\{1,4m-1\}\\
\Lambda & =& \Lambda_{3m}  =  \Lambda_0 +
\overline{\Lambda}_{3m}\nonumber\ .
\end{eqnarray}
Since $\overline{\Lambda}_{3m}$ is the corresponding fundamental
weight of simple Lie algebra $\mathfrak{sl}_{4m}$ we can write (see
\cite{Hum}, pp. 69):
\begin{eqnarray}\label{6.4}
\overline{\Lambda}_{3m} & = &  \frac{1}{4m}[m\alpha_1
+2m\alpha_2+\cdots +(3m-1)\cdot m\alpha_{3m-1} + 3m\cdot
m\alpha_{3m}\nonumber \\
&& +3m\cdot (m-1)\alpha_{3m+1}+3m\cdot(m-2)\alpha_{3m+2}\\
& &  +\cdots +3m\cdot 2\alpha_{4m-2}+3m\cdot
1\alpha_{4m-1}]\nonumber\ .
\end{eqnarray}
The numerator of the formula (\ref{3.13}) looks like
\[
e^{\Lambda_0
+\frac{1}{2}|\Lambda_{3m}|^2\delta}\sum_{\gamma\in\overline{Q}+\overline{\Lambda}_{3m}}e^{\gamma
-\frac{1}{2}|\gamma|^2\delta}
\]
where
\begin{equation}\label{6.5}
 \gamma= k_1\alpha_1
+k_2\alpha_2+\cdots +k_{4m-1}\alpha_{4m-1}+\overline{\Lambda}_{3m}\
.
\end{equation}
Using (\ref{6.4}) the vector $\gamma$ is written down by the base
$\Delta$
\begin{eqnarray}
 \gamma & = &(k_1+\frac{1}{4})\alpha_1
+(k_2+\frac{2}{4})\alpha_2+\cdots
+(k_{3m}+\frac{3m}{4})\alpha_{3m}\nonumber\\
& &
+(k_{3m+1}+\frac{3(m-1)}{4})\alpha_{3m+1}+(k_{3m+2}+\frac{3(m-2)}{4})\alpha_{3m+2}\nonumber\\
& & +\cdots+(k_{4m-2}+\frac{3\cdot
2}{4})\alpha_{4m-2}+(k_{4m-1}+\frac{3\cdot
1}{4})\alpha_{4m-1}\nonumber\ .
\end{eqnarray}
For the partition $\underline{4m}= \{1,4m-1\}$ it is obvious (see
\ref{5.20}) that the number $N$ is equal $4m-1$. Besides, the
mentioned partition implicates that the blockform of $4m\times 4m$
matrices looks like
\[
\left[\begin{array}{l|l}
   {B_{11}}_{\ \ 1\times 1} & {B_{12}}_{\ \ 1\times 4m-1}  \\ \hline
   {B_{21}}_{\ \ 4m-1\times 1} & {B_{22}}_{\ \ 4m-1\times 4m-1}
\end{array}\right]
\]
Following (\ref{5.27}), (\ref{5.28}) and (\ref{5.31}) we can
conclude that the specialization ${\cal F}_{\bf s}$, defined by
\begin{eqnarray}
{\bf s}& = &
(degE_{4m-1,1}^{21}+N\ ,\ degE_{1,1}^{12}\ ,\ degE_{1,2}^{22}\ ,...,\ degE_{4m-2,4m-1}^{22})\nonumber\\
&=& (2m,-2m+1,1,...,1)\nonumber\ ,
\end{eqnarray}
is the specialization for the connection between two standpoints.
More explicitly the specialization is given by
\[
\begin{array}{lll}
e^{-\alpha_0} &\longleftrightarrow& q^{2m}\\
e^{-\alpha_1} &\longleftrightarrow& q^{-2m+1}\\
e^{-\alpha_2} &\longleftrightarrow& q^{1}\\
\vdots && \vdots\\
e^{-\alpha_{4m-1}} &\longleftrightarrow& q^{1}\\
e^{-\delta} &\longleftrightarrow& q^{4m-1}\ .
\end{array}
\]
After  calculation
\[|\gamma|^2 = (\gamma\mid\gamma) = (\ref{6.5}) =2\kappa (k_1,...,k_{4m-1}) + 2k_{3m}+\frac{9m}{4}\]
and the fact that $mult\ n\delta$ always equals $dim
\mathfrak{h}=4m-1$ the left side of the formula (\ref{5.24}) has the
following form
\begin{eqnarray}\label{6.6}
&&{\cal F}_{\bf s} (\frac
{\sum_{\gamma\in\overline{Q}+\overline{\Lambda}}e^{\Lambda_0+\gamma-\frac{1}{2}|\gamma|^{2}\delta}}{\prod_{n\geq
1}(1-e^{-n\delta})^{mult\ n\delta}})\nonumber\\
&=& q^{const}\frac{\sum_{k_1,...,k_{4m-1}}q^{(4m-1)\kappa
(k_1,...,k_{4m-1})+ lin(k_1,...,k_{4m-1})}}{[\varphi
(q^{4m-1})]^{4m-1}}
\end{eqnarray}
where
\[
lin (k_1,...,k_{4m-1})=(2m-1)k_1 -k_2-\cdots -
k_{3m-1}+(4m-2)k_{3m}-k_{3m+1}-\cdots -k_{4m-1}\ .
\]
The right hand side of the formula (\ref{5.24}) for the mentioned
settings (\ref{6.3}) looks like
\begin{eqnarray}
&q^{const}&\prod_{j\geq
1}(1-q^{(4m-1)j})\frac{\sum_{k_1+k_2=3m}q^{\frac{4m-1}{2}(\frac{k_1^2}{1}+\frac{k_2^2}{4m-1})}}{
\prod_{j\geq 1}(1-q^{\frac{(4m-1)j}{1}})\prod_{j\geq
1}(1-q^{\frac{(4m-1)j}{4m-1}})}=\nonumber\\
=&q^{const}&\frac{\sum_{k_1+k_2=3m}q^{\frac{1}{2}[(4m-1)k_1^2+k_2^2]}}{
\varphi(q)}\nonumber\ .
\end{eqnarray}
After the substitution $k_2=3m-k_1$ the calculations
\begin{eqnarray}
\sum_{k_1+k_2=3m}q^{\frac{1}{2}[(4m-1){k_1^2}+{k_2^2}]}&=& \sum_{k_1\in\Bb Z}q^{\frac{1}{2}[(4m-1)k_1^2+(3m-k_1)^2)]}\nonumber\\
&=&q^{\frac{9m^2}{2}}\sum_{k_1\in\Bb Z}q^{\frac{m}{2}(4 k_1^2-6
k_1)}\nonumber\\
&=&q^{\frac{9m^2}{2}}\sum_{k_1\in\Bb Z}q^{m[2 (k_1-1)^2+(k_1-1)-1]}\nonumber\\
&=&q^{\frac{9m^2-2m}{2}}\sum_{k_1\in\Bb Z}q^{m(2
k_1^2+k_1)}\nonumber\\
(Gauss\ \ref{1}) & = & q^{\frac{9m^2-2m}{2}}\cdot\frac{\varphi
{(q^{2m})}^{2}}{\varphi (q^m)}\nonumber\
\end{eqnarray}
implicate that the right hand side of the formula (\ref{5.24}) has
the  form
\begin{equation}\label{6.7}
q^{const}\frac{\sum_{k_1+k_2=3m}q^{\frac{1}{2}[(4m-1)k_1^2+k_2^2]}}{
\varphi(q)}=q^{const}\frac{\varphi {(q^{2m})}^{2}}{\varphi(q)\varphi
(q^m)}\ .
\end{equation}
Now, from (\ref{6.6}) and (\ref{6.7}) it is obvious that the
series-product identity  (\ref{6.1}) holds for (\ref{6.2}).
\end{proof}
\section{THE CLASS $GAUSS(\underline{n}=m+3m,\Lambda_{4m-1})$}
\def\theequation{\thesection.\arabic{equation}}
\setcounter{equation}{0}

Denote again by $\kappa$ the following polynomial of several
variables
\[
\kappa (k_1,...,k_{l})= k_1^2+k_2^2+\cdots +k_{l}^2
-k_1k_2-k_2k_3-\cdots-k_{l-1}k_{l}\ .
\]
\begin{theorem}
Let $n=4m$ for an arbitrary positive integer $m$.  Then the
following series-product identity
\begin{equation}\label{7.1}
\sum_{k_1,...,k_{4m-1}\in\Bb Z}
q^{3m\kappa(k_1,...,k_{4m-1})+\widetilde{lin} (k_1,...,k_{4m-1})} =
\frac{\varphi(q^{3m})^{4m}\varphi
(q^{2})^2}{\varphi(q)^2\varphi(q^3)}
\end{equation}
holds for
\begin{eqnarray}\label{7.2}
\widetilde{lin} (k_1,...,k_{4m-1})&=& -3k_1-\cdots -3k_{m-1} +(3m-2)k_{m} \nonumber\\
&&-k_{m+1}-\cdots -k_{4m-2}+(3m-1)k_{4m-1}\ .
\end{eqnarray}
\end{theorem}
\begin{proof}
First of all, the proof is based on the following setting:
\newpage
\begin{eqnarray}\label{7.3}
\hat\mathfrak{g}& = & \widehat\mathfrak{sl}_{4m}\nonumber\\
\underline{n} & = & \underline{4m}=\{m,3m\}\\
\Lambda & =& \Lambda_{4m-1}  =  \Lambda_0 +
\overline{\Lambda}_{4m-1}\nonumber\ .
\end{eqnarray}
Now,  it is well known (see \cite{Hum}) that
\begin{equation}\label{7.4}
\overline{\Lambda}_{4m-1} =  \frac{1}{4m}[\alpha_1
+2\alpha_2+3\alpha_3 +\cdots
+(4m-2)\alpha_{4m-2}+(4m-1)\alpha_{4m-1}]\ .
\end{equation}
The numerator of the formula (\ref{3.13}) looks like
\[
e^{\Lambda_0
+\frac{1}{2}|\Lambda_{4m-1}|^2\delta}\sum_{\gamma\in\overline{Q}+\overline{\Lambda}_{4m-1}}e^{\gamma
-\frac{1}{2}|\gamma|^2\delta}
\]
where
\begin{equation}\label{7.5}
  \gamma= k_1\alpha_1 +k_2\alpha_2+\cdots
+k_{4m-1}\alpha_{4m-1}+\overline{\Lambda}_{4m-1}\ .
\end{equation}
Using (\ref{7.4}) the vector $\gamma$ is written down by the base
$\Delta$
\[
\gamma =(k_1+\frac{1}{4m})\alpha_1
+(k_2+\frac{2}{4m})\alpha_2+\cdots
+(k_{4m-1}+\frac{4m-1}{4m})\alpha_{4m-1}\nonumber
\]
Now, the number $N$ is equal $3m$ when the partition is
$\underline{4m}= \{m,3m\}$ (see \ref{5.20}). This partition
implicates the blockform of $4m\times 4m$ matrices
\[
\left[\begin{array}{l|l}
   {B_{11}}_{\ \ m\times m} & {B_{12}}_{\ \ m\times 3m}  \\ \hline
   {B_{21}}_{\ \ 3m\times m} & {B_{22}}_{\ \ 3m\times 3m}
\end{array}\right]
\]
Following again (\ref{5.27}), (\ref{5.28}) and (\ref{5.31}) we can
concluded that the demanded specialization ${\cal F}_{\bf s}$ is
defined by
\begin{eqnarray}
{\bf s}& = &
(degE_{3m,1}^{21}+N,degE_{1,2}^{11},...,degE_{m-1,m}^{11},degE_{m,1}^{12},degE_{1,2}^{22},...,degE_{3m-1,3m}^{22})\nonumber\\
&=& (2,3,...,3,2-3m,1,...,1)\nonumber\ .
\end{eqnarray}
More explicitly
\[
\begin{array}{lll}
e^{-\alpha_{0}} &\longleftrightarrow& q^{2}\\
e^{-\alpha_1} &\longleftrightarrow& q^{3}\\
\vdots && \vdots\\
e^{-\alpha_{m-1}} &\longleftrightarrow& q^{3}\\
e^{-\alpha_{m}} &\longleftrightarrow& q^{2-3m}\\
e^{-\alpha_{m+1}} &\longleftrightarrow& q^{1}\\
\vdots && \vdots\\
e^{-\alpha_{4m-1}} &\longleftrightarrow& q^{1}\\
e^{-\delta} &\longleftrightarrow& q^{3m}\ .
\end{array}
\]
After  calculation
\[|\gamma|^2 = (\gamma\mid\gamma) = (\ref{7.5}) =2\kappa (k_1,...,k_{4m-1}) + 2k_{4m-1}+\frac{4m-1}{4m}\]
and the fact that $mult\ n\delta$ always equals $dim
\mathfrak{h}=4m-1$ the left side of the formula (\ref{5.24}) has the
following form
\begin{eqnarray}\label{7.6}
&&{\cal F}_{\bf s} (\frac
{\sum_{\gamma\in\overline{Q}+\overline{\Lambda}}e^{\Lambda_0+\gamma-\frac{1}{2}|\gamma|^{2}\delta}}{\prod_{n\geq
1}(1-e^{-n\delta})^{mult n\delta}})\nonumber\\
&=& q^{const}\frac{\sum_{k_1,...,k_{4m-1}}q^{(3m)\kappa
(k_1,...,k_{4m-1})+ \widetilde{lin}(k_1,...,k_{4m-1})}}{[\varphi
(q^{3m})]^{4m-1}}
\end{eqnarray}
where
\[
\widetilde{lin}  (k_1,...,k_{4m-1})=-3k_1-\cdots -3k_{m-1}
+(3m-2)k_{m}-k_{m+1}-\cdots -k_{4m-2}+(3m-1)k_{4m-1}\ .
\]
The right hand side of the formula (\ref{5.24}) for the mentioned
settings (\ref{7.3}) looks like
\begin{eqnarray}
&q^{const}&\prod_{j\geq
1}(1-q^{(3m)j})\frac{\sum_{k_1+k_2=4m-1}q^{\frac{3m}{2}(\frac{k_1^2}{m}+\frac{k_2^2}{3m})}}{
\prod_{j\geq 1}(1-q^{\frac{(3m)j}{m}})\prod_{j\geq
1}(1-q^{\frac{(3m)j}{3m}})}=\nonumber\\
=&q^{const}\varphi(q^{3m})&\frac{\sum_{k_1+k_2=4m-1}q^{\frac{1}{2}(3k_1^2+k_2^2)}}{
\varphi(q)\varphi(q^3)}\nonumber\ .
\end{eqnarray}
After the substitution $k_2=(4m-1)-k_1$ the calculations
\begin{eqnarray}
\sum_{k_1+k_2=4m-1}q^{\frac{1}{2}(3{k_1^2}+{k_2^2})}&=& \sum_{k_1\in\Bb Z}q^{\frac{1}{2}[3k_1^2+((4m-1)-k_1)^2]}\nonumber\\
&=&q^{\frac{(4m-1)^2}{2}}\sum_{k_1\in\Bb Z}q^{\frac{1}{2}(4
k_1^2-2(4m-1)k_1 )}\nonumber\\
&=&q^{\frac{(4m-1)^2}{2}}\sum_{k_1\in\Bb Z}q^{2 (k_1-m)^2+(k_1-m)-2m^2+m}\nonumber\\
&=&q^{\frac{(4m-1)^2}{2}-2m^2+m}\sum_{k_1\in\Bb Z}q^{2 (k_1-m)^2+(k_1-m)}\nonumber\\
(Gauss\ \ref{1}) & = & q^{\frac{14m^2-7m+1}{2}}\cdot\frac{\varphi
{(q^{2})}^{2}}{\varphi (q)}\nonumber\
\end{eqnarray}
implicate that the right hand side of the formula (\ref{5.24}) has
the  form
\begin{equation}\label{7.7}
q^{const}\varphi(q^{3m})\frac{\sum_{k_1+k_2=4m-1}q^{\frac{1}{2}(3k_1^2+k_2^2)}}{
\varphi(q)\varphi(q^3)}=q^{const}\frac{\varphi(q^{3m})\varphi
{(q^{2})}^{2}}{\varphi(q^3)\varphi (q)^2}\ .
\end{equation}
Now, from (\ref{7.6}) and (\ref{7.7}) it is obvious that the
series-product identity  (\ref{7.1}) holds for (\ref{7.2}).
\end{proof}
\begin{remark}
Notice that for $m=1$ series-product identities (\ref{2}) and
(\ref{3}) (i.e. (\ref{6.1}) and (\ref{7.1})) are the same.
\end{remark}


\begin{thebibliography}{99}
\bibitem{Kac} Kac, V.G.: {\em Infinite dimensional Lie algebras} ($3^{rd}$ edition),\\
Cambridge University Press, (1990).
\bibitem{FL} Feingold, A.J., Lepowsky, J.: The Weyl-Kac character
formula and power series identities, Adv. Math. 29 (1978), 271-309.
\bibitem{K1} Kac, V.G.: Infinite-dimensional Lie algebras and
Dedekind's $\eta$-function, classical M\"{o}bious function and very
strange formula, Adv. Math. 30 (1978), 85-136.
\bibitem{KP1} Kac, V.G., Peterson, D.H.: Infinite-dimensional Lie algebras, theta functions and modular forms,
Adv. Math. 53 (1984), 125-264.
\bibitem{KL} ten Kroode, F., van de Leur J.: Bosonic and fermionic realization of the affine algebra
$\widehat{\mathfrak{gl}}_n$, Comm.Math.Phys. 137 (1991), 67-107.
\bibitem{FK} Frenkel, I.B., Kac V.G.: Basic representation of affine
Lie algebras and dual resonance models, Invent. Math. 62 (1980),
23-66.
\bibitem{KP2} Kac, V.G., Peterson, D.H.: 112 construction of the basic representation of the loop group of $E_8$,
Proceedings of the conference "Anomalies, geometry,  topology"
Argone, 1985. World Sci. (1985), 276-298.
\bibitem{K3} Kac, V.G.: Automorphisms of finite order of semi-simple Lie
algebras, Funct.Anal.Appl. 3 (1969), 252-254.
\bibitem{TS2} \v{S}iki\' c T.: $\Bb Z$-Gradations of Classical Affine Lie Algebras and Kac Parameters,
Comm. in Algebra Vol. 32, No. 8 (2004),  2987-3016.
\bibitem{Hum} Humphreys, J.E.: {\em Introduction to Lie algebras and Representation Theory}.\\ NY, Heidelberg, Berlin: Springer-Verlag,
(1972).
\end{thebibliography}
\end{document}